\documentclass{article}
\usepackage{mathrsfs}
\usepackage{tikz-cd}
\usepackage{amsmath}

\usepackage[toc,page]{appendix}

\usepackage{enumerate}
\usepackage{amsfonts}

\usepackage[colorlinks=true,citecolor=red,linkcolor=blue,
urlcolor=RubineRed,pdfpagetransition=Blinds,pdftoolbar=false,pdfmenubar=false]{hyperref}
\usepackage{color}
\hypersetup{
colorlinks=true,
linkcolor=blue,
urlcolor=blue,
pdfpagemode=FullScreen,
}
\usepackage[top= 3cm, bottom= 3cm, right= 3 cm, left=3cm]{geometry}

\everymath{\displaystyle}
\usepackage{amsthm}
\usepackage{tikz}
\usetikzlibrary{positioning}
\usetikzlibrary{patterns}
\usetikzlibrary{decorations.markings}
\usetikzlibrary {decorations.pathmorphing}
\usetikzlibrary{shapes.geometric}
\date{} 
\newtheorem{teo}{Theorem}[section]
\newtheorem{defi}[teo]{Definition}

\newtheorem{propo}[teo]{Proposition}
\newtheorem{obs}[teo]{Observation}

\newtheorem{lema}[teo]{Lemma}

\title{About an isomorphism between the Beurling algebra with a weight-dependent convolution and the $L^1(G)$ group algebra}
\author{Raúl Rodríguez-Barrera \\ Francisco Torres-Ayala}

\begin{document}
\maketitle

\abstract{We show that the Beurling algebra with a weight-dependent convolution and the group algebra $L^1(G)$ are isomorphic. In particular, using this isomorphism, we extend some results of the algebra $\mathscr{L}^1(G,\omega)$ presented in recent articles. As a main result, we explicitly construct the equivalence between unitary representations of the group and non-degenerate $\ast$-representations of this algebra. }\\


\section*{Introduction}
In abstract harmonic analysis, given a locally compact group $G$ with a fixed left (or right) Haar measure $\mu$,
one of the classical structures of study are the Banach spaces $L^1(G,\mathbb{B}_G,\mu)$ or simply  $L^1(G)$ (where $\mathbb{B}_G$ denotes the $\sigma$-algebra
of Borel sets in $G$), considering them  as Banach $\ast$-algebras with the product given by the usual convolution
\begin{equation*}
 f\ast g (t)=\int_G f(s)g(s^{-1}t)d\mu(s),
\end{equation*}
and the involution given by
\begin{equation*}
 f^*(t)=\Delta(t^{-1})\overline{f(s^{-1})},   
\end{equation*}
where $\Delta$ denotes the modular function of $(G,\mu)$, standard references are \cite{Deitmar},\cite{Folland}, and \cite{hewitt}.

This turns out to be a fruitful point of view because some structural aspects of $G$ are strongly related with similar ones of \(L^1(G)\). Two examples: it is a well known result that unitary representations of \(G\)  are in a one-to-one correspondence with the  
non-degenerate \(*\)-representatios of \(L^1(G)\);
amenability in the sense of groups is equivalent in the sense of Banach algebras, see \cite{gronbaek}.

Now, as Banach spaces a natural generalization of \(L^1\) spaces are the weighted  \(L^1\) spaces
(for instance they occur in regular Sturm-Liouville problems). In this frame one starts with a measurable function \(\omega: G \to (0,\infty)\)
and consider the Banach spaces \(L^1(G,\mathbb{B}_G, \omega d\mu)\)  (simply denoted  \(L^1(G,\omega)\)) as a Banach algebra  with the same
convolution and same involution as in \(L^1(G)\). This produces a new theory, the theory of Beaurling algebras (introduced in \cite{Beurling}) in which results are more delicate than
those of \(L^1(G)\), for instance, the equivalence of amenability of \(G\) and amenability of \(L^1(G,\omega)\) is no longer valid unless
more conditions on the weight function \(\omega\) are required, see \cite{gronbaek}.

In this paper we are concerned with the Banach spaces \(L^1(G,\omega)\) but with a different algebra structure. This new structure
was introduced in \cite{Mahmoodi} by A. Mahmodi via a new weight dependent convolution. The advantage of this new convolution is that similar results
for $L^1(G)$ are now also true for $L^1(G,\omega)$. Note: to prevent confusion we will be denoting $L^1(G,\omega)$ 
with the new product as $\mathscr{L}^1(G,\omega)$. The paper is divided in two parts. In the first part we introduce an isomorphism between Banach algebras that brings light to the new product defined by A. Mahmodi and simplifies the proofs for the rest of the article, also we introduce the generalization of right translation operator and we study some of the properties of the generalization of the left translation introduced in \cite{multipliers} by A.Issa, and Y. Mensa. In the second part we prove the equivalence of unitary representations of $G$ and non-degenerate $*$-representations of $\mathscr{L}^1(G,\omega)$ (which is only done by discrete groups in \cite{Mahmoodi}). 

\section{Beurling Algebra}

\begin{defi}
    Let $G$ be a locally compact topological group, a Borel measurable function $\omega: G\rightarrow (0,+\infty)$ is called a weight if it satisfies
    \begin{enumerate}
        \item $\omega(st)\leq\omega(s)\omega (t)$ for all $s,t\in G,$
        \item $\omega(e)=1$ where $e$ is the neutral element of $G$.
    \end{enumerate} 
\end{defi}
We denote by 
    \begin{equation*}
        \mathcal{B}(G):=\{\textrm{$f:G\to\mathbb{C}$ : such that $f$ is Borel measurable}\},
    \end{equation*}
 and
    \begin{equation*}
    L^1(G,\omega):=\{f\in\mathcal{B}(G)\ :\ \int_{G}^{}|f(s)|\omega(s)d\mu(s)<+\infty\}     
    \end{equation*}
where $\mu$ is a left Haar measure on $G.$ The idea of introducing these spaces is to adapt some of the techniques of harmonic analysis to a broader range of functions because the space $L^1(G)$ and $L^1(G,\omega)$ share some similarities but in some specific aspects they can be very different, depending on the weight function, even though there is an easy relation between them: $f\in L^1(G,\omega)$ if only if $f\omega\in L^1(G)$. Regarding similarities, the first result that can be consulted in \cite[Section 1.3]{Mahmoodi} is as follows.
    \begin{teo}
        With the classical convolution and the norm given by
\begin{equation*}
     \|f\|_{\omega,1}:=\int_{G}^{}|f(s)|\omega(s)d\mu(s),
\end{equation*}
the space $L^1(G,\omega)$ has a Banach algebra structure, called the Beurling algebra of G. 
    \end{teo}
Of course, with the involution given by $f^\ast(s)=\Delta(s^{-1})\overline{f(s^{-1})}$  it is satisfied that $L^1(G,\omega)$ has $\ast$-Banach algebra structure. However, unlike in the usual case, the equivalence between the unitary representations of the group and the non-degenerate $\ast$-representations of the algebra $L^1(G,\omega )$ is not satisfied. With this in mind, A.Mahmoodi introduced in \cite{Mahmoodi} a {\it new}  weight dependent convolution  to give a {\it new} Banach algebra structure to the Banach 
space $L^1(G,\omega)$.

The work of A. Mahmoodi is divided into two sections. In the first one he proves that if $G$ is a discrete group, under certain conditions on the weight function, one has an equivalence between the representations of the group and the representations of the {\it new}  Banach algebra, see \cite[Theorem 3.1] {Mahmoodi}. In the following, we give the definition of this {\it new} weight dependent convolution.
\begin{defi}
 Let $G$ be  a locally compact group, $\mu$ be a left Haar measure on $G,$ and $\omega$ be a weight on $G.$ For elements $f$ and $g$ in $L^1(G,\omega)$ a  weight dependent convolution, denoted by $\ast_\omega$, is defined  as 
    \begin{equation*}
     (f\ast_\omega g)(t)=\int_{G}^{}{f(s)g(s^{-1}t)\frac{\omega (s)\omega (s^{-1}t)}{\omega (t)}d\mu(s)}.
    \end{equation*}
   
\end{defi}
 With this convolution, A. Mahmoodi shows that properties analogous to those for the $L^1(G)$ group algebra are satisfied as well. For example, the analog to Young's inequality for convolutions holds; see \cite[Theorem 2.1]{Mahmoodi}.
 
 \begin{propo}
     Let $G$ be a locally compact group, $\omega$ be a weight on $G$ and $f,g\in L^1(G,\omega).$ Then $f\ast_\omega g\in L^1(G,\omega)$, furthermore
     $$||f\ast_\omega g||_{\omega,1}\leq||f||_{\omega,1}||g||_{\omega,1}.$$
 \end{propo}
 To endow the Beurling algebra with $\ast$-Banach algebra structure, further conditions on the weight function are necessary.

 \begin{defi}\label{simetrico}
      Let $G$ be a locally compact group, and $\omega$ be a weight on $G$. The weight is said to be symmetric if we have $\omega(s^{-1})=\omega(s)$ for each $s$ in $G.$
 \end{defi}
 Under the above weight condition, one has the following result, see \cite[Theorem 2.1]{Mahmoodi} and the discussion of \cite[Page 56]{Folland}.
 
 \begin{teo}
     Let $G$ be a locally compact group and $\omega$  a symmetric weight in $G.$ Then the Banach space $(L^1(G,\omega), \|\cdot\|_{\omega,1})$ has a $\ast$-Banach algebra structure with the weight dependent convolution as a product, the usual sum and the usual involution:  $f^*(s)=\overline{f(s^{-1})}\Delta(s^{-1}).$ We denote this algebra by  $\mathscr{L}^1(G,\omega)$, that is $$\mathscr{L}^1(G,\omega)=(L^1(G,\omega),\ast_\omega,+,\|\cdot\|_{\omega,1}).$$ 
 \end{teo}
 \subsection{About an Isomorphism Between Banach Algebras}
 In this section, we will focus on comparing the structure of the Banach algebras $L^1(G), L^1(G,\omega)$ and $\mathscr{L}^1(G,\omega)$. The first result we obtained is as follows.
  
\begin{lema}\label{sigma es iso}
    Let $G$ be a locally compact group, $\mu$ be a left Haar measure on $G,$ and $\omega$ be a symmetric weight in  $G.$ We define $\sigma: \mathscr{L}^1(G,\omega)\to L^1(G)$ by $\sigma(f)=f\omega,$ then $\sigma$ is an isometric $\ast$-isomorphism of Banach algebras.
\end{lema}
\begin{proof}
   Let $f,g\in\mathscr{L}^1(G,\omega)$. Directly from the definition of $*_\omega$ we have that $f*_\omega g=\omega^{-1}(f\omega)*(g\omega)$, hence:
   \begin{align*}
   \sigma(f\ast_\omega g)&=(f\ast_\omega g)\omega=\frac{f\omega\ast g\omega}{\omega}\omega=\sigma(f)\ast\sigma(g).
   \end{align*}
    On the other hand, given $h\in L^1(G)$ it follows that $\tfrac{h}{\omega}\in\mathscr{L}^1(G,\omega)$ and $\sigma (\tfrac{h}{\omega})=h$, so $\sigma$ is surjective. Since $\omega >0$ the injectivity of $\sigma$ follows. Furthermore, $\|f\|_{\omega,1}=\|f\omega\|_1=\|\sigma(f)\|_1.$
Finally,
\begin{equation*}
   \sigma(f^\ast)(s)=f^\ast(s)\omega(s)=\Delta(s^{-1})\overline{f(s^{-1})}\omega(s^{-1})=\Delta(s^{-1})\overline{\sigma(f)(s^{-1})}=\sigma(f)^\ast(s).\qedhere
\end{equation*}
\end{proof}
In \cite[Section 3]{multipliers} A.Issa, and Y. Mensah introduced the following operator.
\begin{defi}
Let $G$ be a locally compact group and $\omega$ a weight on $G$. For each $s\in G$ and $f\in \mathscr{L}^1(G,\omega)$ the operator $\Gamma_\omega^s$ is defined as follows
\begin{equation*}
 \left(\Gamma^s_\omega f\right)(t)=\frac{L_s(\omega f)}{\omega}(t)=\frac{f\left(s^{-1} t\right) \omega (s^{-1} t)}{\omega(t)}.
\end{equation*}
\end{defi}
 The above isomorphism clarifies many parts of the work developed by Y.Mensah and A.Issa, see \cite{multipliers}. For this, let us consider the following diagram
\begin{center}
\begin{tikzcd}
\mathscr{L}^{1}(G,\omega)\ar{r}{\sigma}\ar{d}[swap]{\Gamma_{\omega}^s}
& L^{1}(G)\ar{d}{L_s}\\
\mathscr{L}^{1}(G,\omega) & L^{1}(G).\ar{l}{\sigma^{-1}}
\end{tikzcd}
\end{center}  
Note that $\Gamma_{\omega}^s$ is the mapping that makes the diagram commute. We can infer that $s\to  \Gamma_{\omega}^s$ is the mapping of $G$ corresponding to the mapping $s\to L_s$.  It should be noted that in their work, they did not introduce the mapping corresponding to the right translation, which we define below.

\begin{defi}
Let $G$ be a locally compact group and $\omega$  a weight on $G$. For each $s \in G$ and $f \in \mathscr{L}^1(G, \omega)$, we define the operator  $\Theta_\omega^s$ as follows

\begin{equation*}
    (\Theta_\omega^s f)(t) = \frac{R_s(\omega f)}{\omega}(t) = \frac{f(t s) \omega(t s)}{\omega(t)}.
\end{equation*}
\end{defi}
Similarly, we have that $\Theta_\omega^s = \sigma^{-1} R_s \sigma$, where $R_s(f)(t) = f(st)$ is the right translation. Now, for $1 \leq p < +\infty$, we consider the spaces
\begin{equation*}
L^p(G,\omega):=\{ f\in\mathcal{B}(G)\ :\ \int_G |f(s)|^p\omega(s)d\mu(s)<+\infty\}   
\end{equation*} 
with the norm
\begin{equation*}
 \||f\||_{\omega,p}:=\left(\int_G |f(s)|^p\omega(s)d\mu(s)\right)^\frac{1}{p}.   
\end{equation*}
 As in the standard case, the following observation is immediate.
\begin{obs}
If we take $s=e,$ for each $f\in L^p(G,\omega)$ it follows that 
\begin{equation*}
    \left(\Gamma^e_\omega f\right)(t)=f(t)= \left(\Theta^e_\omega f\right)(t).
\end{equation*}
\end{obs}

In \cite[Theorem 2.3.]{OnBeurling} the authors showed that if $G$ is a locally compact abelian topological group, the operator $\Gamma_\omega^s$ satisfies the following: for each $f \in L^p(G, \omega)$
\begin{equation*}
{\omega(s)}^{\frac{1-p}{p}} |||f|||_{\omega,p} \leq |||\Gamma_\omega^s f|||_{\omega,p} \leq \omega(s^{-1})^{\frac{p-1}{p}} |||f|||_{\omega,p}.
\end{equation*}
For $\Theta^s_\omega$ we obtained the following result.
\begin{propo}
Let $G$ be a locally compact abelian group, $\omega$ a weight on $G,$  $f \in L^p(G,\omega).$Then, for each $s\in G$ it holds that
\begin{equation*}
  {\omega(s^{-1})}^{\frac{1-p}{p}}|||f|||_{ \omega,p}\leq |||\Theta_\omega^s f|||_{\omega,p} \leq\omega(s)^{\frac{p-1}{p}}|||f|||_{\omega,p}.  
\end{equation*}
  
\end{propo}

\begin{proof}
Let $f \in L^p(G,\omega),$ for each $s \in G$, we have that

\begin{align*}
||| \Theta_\omega^s f |||_{\omega}^p& =\int_G|f(ts)|^p\left(\frac{\omega\left(ts\right)}{\omega(t)}\right)^p \omega(t) d\mu(t) \\
& =\int_G|f(z)|^p\left(\frac{\omega(z)}{\omega(s^{-1} z)}\right)^{p-1} \omega(z) d \mu(z),
\end{align*}
but $\omega(z)=\omega\left(s^{-1} s z\right) \leq \omega\left(s^{-1}z\right) \omega(s )$. Then $\frac{\omega(z)}{\omega(s^{-1} z)} \leq \omega(s)$, it follows that
\begin{align*}
|||\Theta_\omega^s f |||_{ \omega}^p & \leq \int_G|f(z)|^p\left[\omega(s)\right]^{p-1} \omega(z) d \mu(z)=[\omega(s)]^{p-1}|||f|||_{ \omega}^p .
\end{align*}
In the opposite direction, we have that
\begin{align*}
|||\Theta_\omega^s f|||_{\omega}^p& =\int_G\left|f( ts)\right|^p\left(\frac{\omega\left(ts\right)}{\omega(t)}\right)^p \omega(t) d \mu(t) \\
& =\int_G\left|f(ts)\right|^p\left(\frac{\omega(ts)}{\omega(t)}\right)^{p-1} \omega(ts) d \mu(t)
\end{align*}
Given that $ \omega(ts)\omega(s^{-1}) \geq \omega(t)$. It follows that $\frac{\omega(ts)}{\omega(t)} \geq\frac{1}{\omega(s^{-1})}$. Therefore
\begin{align*}
|||\Theta_\omega^s f|||_{\omega}^p & \geq  \int_G|f(ts)|^p 
\left(\frac{1}{\omega(s^{-1})}\right)^{p-1}\omega(ts) d \mu(t) \\
& \geq\frac{1}{\omega(s^{-1})^{p-1}}|||f|||_{\omega}^p .
\end{align*}
Hence,
\begin{equation*}
 {\omega(s^{-1})}^{\frac{1-p}{p}}|||f|||_{ \omega}\leq |||\Theta_\omega^s f|||_{\omega} \leq\omega(s)^{\frac{p-1}{p}}|||f|||_{\omega}.  \qedhere 
\end{equation*}
  
\end{proof}
Next, we examine some of the properties of the previous operators.
\begin{lema}\label{gama y theta brincan }
Let $G$ be a locally compact abelian group, $\omega$ a symmetric weight on $G$, and $f, g \in \mathscr{L}^1(G, \omega)$. Then, the following identities  hold
\begin{enumerate}
\item $\Gamma_\omega^s(f \ast_\omega g)=f \ast_\omega \Gamma_\omega^s g=\Gamma_\omega^s f \ast_\omega g,$
\item $\Theta_\omega^s(f \ast_\omega g)=f \ast_\omega \Theta_\omega^s g=\Theta_\omega^s f \ast_\omega g.$
   
\end{enumerate}
\end{lema}
\begin{proof} 
\begin{enumerate}
\item See \cite[Proposition 3.4]{multipliers}.
\item This statement can be proven by following the reasoning in \cite[Proposition 3.4]{multipliers}; however, the isomorphism $\sigma$ simplifies the proof. From \cite[Lemma 1.6.3]{Deitmar}, we know that $R_s h \ast k = h \ast R_s k$ for each $h, k \in L^1(G)$. It follows that
\begin{align*}
    \Theta_\omega^s(f\ast_\omega g)&=\sigma^{-1}R_s\sigma (\frac{f\omega\ast g\omega}{\omega})\\
    &=\sigma^{-1}(f\omega\ast R_s(g\omega))\\
     &=\sigma^{-1}(f\omega\ast R_s\sigma(g))\\
     &=\sigma^{-1}(f\omega\ast (\sigma^{-1}
     R_s\sigma(g))\omega)\\
     &=f\ast_\omega \Theta_\omega^sg.\qedhere  
\end{align*}

\end{enumerate}
\end{proof}

\begin{lema}\label{gamma es morfismo}
Let $G$ be a locally compact abelian group, $\omega$ a weight on $G$, and $s, r \in G$. Then
\begin{enumerate}
    \item $\Theta_\omega^s\Theta_\omega^r=\Theta_\omega^{sr},$
    \item $\Gamma_\omega^s\Gamma_\omega^r=\Gamma_\omega^{sr} .$
\end{enumerate}
\end{lema}
\begin{proof}
\
\begin{enumerate} 
\item Let $f,g\in\mathscr{L}^1(G,\omega)$, we have that 
    \begin{equation*}
f\ast_\omega\Theta_\omega^s\Theta_\omega^rg=f\ast_\omega(\sigma^{-1}R_s\sigma)(\sigma^{-1}R_r\sigma)g=f\ast_\omega(\sigma^{-1}R_{sr}\sigma)g=f\ast_\omega\Theta_\omega^{sr}g
\end{equation*}
This implies that $\Theta_\omega^s \Theta_\omega^r = \Theta_\omega^{sr}$.

\item The proof for the operator $\Gamma_\omega^s$ is analogous.
\qedhere
\end{enumerate}

\end{proof}

\begin{lema}\label{Gamma SOT continuo }
Let $f \in L^1(G, \omega)$ be fixed. Then, the mappings $s \mapsto \Gamma_\omega^s f$ and $s \mapsto \Theta_\omega^s f$ from $G$ to $L^1(G, \omega)$ are continuous.
\end{lema}
\begin{proof}
Note that
\begin{align*}
\|\Theta_\omega^s f - \Theta_\omega^r f\|_{\omega, 1} &= \int_G \frac{1}{\omega(t)} |R_s(f\omega)(t) - R_r(f\omega)(t)| \omega(t) \, dt \\
&= \int_G |R_s(f\omega)(t) - R_r(f\omega)(t)| \, dt \\
&= \|R_s(f\omega) - R_r(f\omega)\|_1 \rightarrow 0,
\end{align*}
where the last convergence follows from the fact that $f\omega \in L^1(G)$ and 
the continuity of the right translation in $L^1(G)$, see \cite[Proposition 2.42]{Folland}. The proof for the mapping $s \mapsto \Gamma_\omega^s f$ is analogous.
\end{proof}
\begin{defi}
Let $G$ be a locally compact abelian group, we define the Pontryagin dual of $G,$ denoted by $\widehat{G},$ as follows
    \begin{equation*}
        \widehat{G}:=\{\chi:G\to\mathbb{T}\ : \ \textrm{ $\chi$ is a continuous homomorphism}\}.   
      \end{equation*}
\end{defi}
Next, we will revisit the definition of the Fourier transform introduced in \cite[Definition 5.1]{multipliers}.
\begin{defi}
Let $G$ be a locally compact abelian group, and for a fixed $f \in \mathscr{L}^{1}(G, \omega)$, the $\omega$-Fourier transform is defined as follows
\begin{equation*}
 \widehat{f}_{\omega}(\gamma) = \int_G f(s) \overline{\gamma}(s) \omega(s) \, d\mu(s), \quad \gamma \in \widehat{G}.   
\end{equation*}
\end{defi}
It is worth noting that, this Fourier transform is simply
\begin{equation*}
\widehat{f}_\omega(\gamma)=\widehat{(f\omega)}(\gamma)=\widehat{\sigma(f)}(\gamma)    
\end{equation*}
where the transform in the second equality denotes the usual Fourier transform. Of course, analogous properties to those satisfied in the standard case hold as well; the following are some of them.
 \begin{lema}\label{fourier morfismo}
For $f, g \in \mathscr{L}^{1}(G, \omega)$, it holds that 
\begin{equation*}
 (f \ast_{\omega} g\widehat{)_{\omega}}=\widehat{f}_{\omega} \widehat{g}_{\omega} .   
\end{equation*}
\end{lema}
\begin{proof}
Let $f,g\in\mathscr{L}^1(G,\omega),$ we have that
\begin{equation*}
 (f \ast_{\omega} g\widehat{)_{\omega}}=\sigma(f \ast_{\omega} g\widehat{)}=(\sigma(f) \ast \sigma(g)\widehat{)}=\widehat{\sigma(f)}\widehat{\sigma(g)}=\widehat{f}_{\omega} \widehat{g}_{\omega} .\qedhere   
\end{equation*}
\end{proof}
Given a Banach algebra $\mathcal{A}$, we denote by $\Delta(\mathcal{A})$ the set of multiplicative linear functionals on $\mathcal{A}$. We obtain the following result.
\begin{teo}\label{iso}
The mapping $d: \widehat{G} \to \Delta\left(\mathscr{L}^{1}(G, \omega)\right)$, given by $\gamma \mapsto d_{\gamma, \omega}$, where $d_{\gamma, \omega}(f) = \widehat{f}_\omega(\gamma)$, is a homeomorphism. 
\end{teo}

 \begin{proof}
We know that the mapping $\gamma \mapsto d_{\gamma}$, with $d_\gamma(f) = \widehat{f}(\gamma)$, is a homeomorphism between $\widehat{G}$ and $\Delta(L^1(G))$, see \cite[Theorem 3.2.1]{Deitmar}. Furthermore,
since \(\sigma: \mathscr{L}^{1}(G, \omega) \to L^1(G)\) is an isomorphism of 
\(*\)-Banach algebras it induces a homeomorphism \(\Sigma:\Delta(L^1(G)) \to \Delta(\mathscr{L}^{1}(G, \omega))\) given by \(\Sigma(\varphi)=\varphi \circ \sigma\). A straightforward computation shows that \(\Sigma(d_\gamma)=d_{\gamma, \omega}\)
hence proving that the map \(\gamma \to d_{\gamma, \omega}\) is a homeomorphism. 
 \end{proof}

To conclude this section, we examine some properties of the operators $\Gamma_\omega^s$ and $\Theta_\omega^s$ that will be useful in the following section.

\begin{lema}\label{gama y theta brincan adjuntos }
Let $G$ be a locally compact abelian group, $\omega$ a symmetric weight on $G,$ and $f,g\in\mathscr{L}^1(G,\omega)$, then
\begin{enumerate}
\item $g^{*} \ast_{\omega} \Gamma_{\omega}^{s} f=\left(\Gamma^{s^{-1}}_\omega g\right)^{*} \ast_\omega f,$
    \item $g^{*} \ast_{\omega} \Theta_{\omega}^{s} f=\left(\Theta^{s^{-1}}_\omega g\right)^{*} \ast_\omega f.$ 
\end{enumerate}
\end{lema}
\begin{proof} \
\begin{enumerate}
    \item A straightforward calculation shows that $h^* \ast L_s k = (L_{s^{-1}} h)^* \ast k$ for each $h, k \in L^1(G)$. Hence we have that 
\begin{align*}
 g^*\ast_\omega\Gamma_\omega^s f 
 &=\frac{(g\omega)^\ast\ast L_s (f\omega)}{\omega}\\
  &=\frac{(L_{s^{-1}}(g\omega))^\ast\ast (f\omega)}{\omega}\\
  &=(\Gamma_\omega^{s^{-1}}g)^\ast\ast_\omega f.
\end{align*}
The second equality holds since $\omega$ is a symmetric weight.
\item The proof is analogous to the previous part.
\qedhere
\end{enumerate}

\end{proof}

\section{Correspondence Between Representations }
Recall that a unitary representation of a topological group $G$ is a homomorphism $\pi: G \rightarrow \mathcal{U}(H_\pi)$, where $\mathcal{U}(H_\pi)$ denotes the set of unitary operators on $H_\pi$, which is $SOT$-continuous. That is, the mappings $s \mapsto \pi(s)\xi$ are continuous with respect to the norm for each $\xi \in H_\pi$. On the other hand, a $\ast$-representation of a Banach $\ast$-algebra $\mathcal{A}$ on a Hilbert space $H_\pi$ is a $\ast$-homomorphism $\pi: \mathcal{A} \rightarrow \mathcal{B}(H_\pi)$. Moreover, we say that $\pi$ is non-degenerate if there is a non-zero vector $\xi$ in $ H_\pi$ such that $\pi(a)\xi = 0$ for every $a$  in  $\mathcal{A}$. Equivalently, $\pi(\mathcal{A})H_\pi$ is dense in $H_\pi$. Since the unitary representations of $G$ correspond to the non-degenerate $\ast$-representations of $L^1(G)$, see \cite[Proposition 6.2.1, Proposition 6.2.3]{Deitmar}, and $\sigma$ is an isometric $\ast$-isomorphism, see Lemma \ref{iso}, must be satisfied that the unitary representations of $G$ correspond to the non-degenerate $\ast$-representations of $\mathscr{L}^1(G, \omega)$. It is important to note that this holds not only in the case where $G$ is discrete, as is the main result in the first section of A. Mahmoodi's work, see \cite[Theorem 3.1]{Mahmoodi}. More generally, it is satisfied when $G$ is locally compact. On the other hand, although we know from the previous discussion that an equivalence holds between such representations, we may ask: What structure do the non-degenerate $\ast$-representations of $\mathscr{L}^1(G, \omega)$ have? Through the operators $\Theta_\omega^s$ and $\Gamma_\omega^s$, we were able to explicitly give the structure of these representations.

Let $G$ be a locally compact abelian group. In this section, if $\omega$ is a weight on $G$, we will assume it is symmetric. Given a unitary representation $(\pi, H_\pi)$ of $G$, it is associated with a $\ast$-representation of $\mathscr{L}^1(G, \omega)$, also denoted by $\pi$, as follows. Consider the mapping $A: H_\pi \times H_\pi \rightarrow \mathbb{C}$ given by: 

\begin{equation*}
      A(\xi,\eta)=\int_{G}^{}f(s)\langle \pi(s)\xi,\eta \rangle\omega(s)d\mu(s)
 \end{equation*}
for each $\xi, \eta \in H_\pi$ and $f \in \mathscr{L}^1(G, \omega)$. Note that $A$ is a bounded sesquilinear form since:
 \begin{align*}
     |A(\xi,\eta)|&=\left| \int_{G}^{}f(s)\langle \pi(s)\xi,\eta \rangle\omega(s)d\mu(s) 
 \right|\\
&\leq\|f\|_{\omega,1}\|\xi\|\|\eta\|.
 \end{align*}
Then, by the Riesz representation theorem, there exists a unique operator, denoted by $\pi(f)$, such that
\begin{equation}\label{pi(f)}
 \langle \pi(f)\xi,\eta \rangle= \int_{G}^{}f(s)\langle \pi(s)\xi,\eta \rangle\omega(s)d\mu(s) .   
\end{equation}
Moreover, from  \eqref{pi(f)}, we have that $\|\pi(f)\| \leq \|f\|_{\omega,1}$, and this operator satisfies
\begin{equation*}
\pi(f)=\int_{G}^{}f(s)\pi(s)\omega(s)d\mu(s),    
\end{equation*}
where the right-hand side of the equality is understood as a Bochner integral.

\begin{teo}
Let  $G$  be a locally compact group, $\omega$ a symmetric weight on $G$, and $\pi$ a unitary representation of $G$. Then, the mapping  $f \mapsto \pi(f)$,
where the operator $\pi(f)$ is defined as  
\begin{equation*}
\langle \pi(f)\xi, \eta \rangle = \int_{G} f(s) \langle \pi(s)\xi, \eta \rangle \omega(s) \, d\mu(s),     
\end{equation*}
 is a $\ast$-representation, non-degenerate, of the algebra $\mathscr{L}^1(G, \omega)$ on $H_\pi$. Moreover, for each $ s \in G$  and $f \in \mathscr{L}^1(G, \omega)$, the following equalities hold  
\begin{enumerate}
\item $\pi(s)\pi(f)=\pi\left(\Gamma_\omega^sf\right),$
\item $\pi(f)\pi(s)=\Delta(s^{-1})\pi(\Theta_\omega^{s^{-1}}f).$
\end{enumerate}

\end{teo}
 \begin{proof}
Let $f\in \mathscr{L}^1(G, \omega)$ and $s\in G$, so
\begin{align*}
\left\langle \xi, \pi(f^*) \eta\right\rangle & =\overline{\left\langle\pi(f^*) \eta, \xi\right\rangle}\\
&=\overline{\int_Gf^*( s)\langle\pi(s) \eta, \xi\rangle \omega(s) d \mu(s)} \\
& =\int_G f(z)\langle\pi(z) \xi, \eta\rangle \omega(z) d \mu(z)\\
&=\langle\pi(f) \xi, \eta\rangle,
\end{align*}
then $\pi(f)^*=\pi(f^*).$ On the other hand, we have that
\begin{align*}
 \langle\pi(f) \pi(s) \xi, \eta\rangle&=\int_G f(r)\langle\pi(r) \pi(s) \xi, \eta\rangle \omega(r) d\mu( r)\\
 &=\Delta(s^{-1})\int_G f(z s^{-1}) \omega(z s^{-1})\left\langle\pi(z)\xi,\eta\right\rangle d\mu( z)\\
&=\Delta(s^{-1})\int_G \left(\Theta_\omega^{s^{-1}}f\right)(z)\left\langle\pi(z) \xi, \eta\right\rangle \omega(z) d\mu(z)\\
&=\left\langle \Delta(s^{-1})\pi\left(\Theta_\omega^{s^{-1}}f\right) \xi,\eta\right\rangle.
\end{align*}
Thus $\pi(f)\pi(s)=\Delta(s^{-1})\pi\left(\Theta_\omega^{s^{-1}}f\right).$ Similarly,  $\pi(s) \pi(f)=\pi\left(\Gamma_{\omega}^{s} f\right)$. Next, we have that:
\begin{align*}
\left\langle\pi(f) \pi(g) \xi,\eta\right\rangle&=\int_G f(s) \omega(s)\langle\pi(s) \pi(g) \xi, \eta\rangle d\mu( s) \\
& =\int_G f(s) \omega(s)\left\langle\pi\left(\Gamma_\omega^s g\right) \xi, \eta\right\rangle d \mu(s) \\
& =\int_G f(s) \omega(s)\left(\int_G\left(\Gamma_\omega^s g\right)(t) \omega(t)\langle\pi(t) \xi, \eta\rangle d \mu(t)\right) d\mu( s) \\
& =\int_G(f \ast_\omega g)(t) \omega(t)\langle\pi(t) \xi, \eta\rangle d\mu(t)\\
&=\langle\pi(f \ast_\omega g) \xi, \eta\rangle. \\
\end{align*}
Thus, $\pi(f)\pi(g)=\pi(f\ast_\omega g)$. Finally, let 
$\xi \in H_\pi, \xi \neq 0$ and  let $V$ be a compact neightbord of $e$ in $G$ such that $\|\pi(s) \xi-\xi\|<\|\xi\|$, for each $s\in V$, see \cite[Lemma 6.2.2]{Deitmar}. We consider
$f=(\mu(V))^{-1}\frac{\chi_V}{\omega }$, then $f\in\mathscr{L}^1(G,\omega),$ $\|f\|_{\omega,1}=1$ and, considering $\pi(f)$ as a Bochner integral, we have that
\begin{equation*}
\|\pi(f) \xi-\xi\|=\frac{1}{\mu(V)}\left\|\int_V(\pi(s) \xi-\xi) d s\right\|<\|\xi\|    
\end{equation*}
in particular  $\pi(f)\xi \neq 0,$ wich concludes the proof.
\end{proof}
In the other direction, we obtained the following result.
\begin{teo}
Let $G$ be a locally compact group, $\omega$ a symmetric weight on $G,$ $\pi: \mathscr{L}^{1}(G, \omega) \to \mathcal{B}(H _\pi)$ a non-degenerate $\ast$-representation. Then, there exists a unique unitary representation $\tilde{\pi}: G \to \mathcal{B}(H_\pi)$ such that:
\begin{equation*}
\langle\pi(f) \xi, \eta\rangle=\int_{G} f(s)\langle\tilde{\pi}(s) \xi, \eta\rangle \omega(s) d \mu(s)    
\end{equation*}
for each  $f \in \mathscr{L}^{1}(G, \omega)$ and each  $\xi,\eta \in H_{\pi}$.
\end{teo}
\begin{proof}
Firstly, we prove the existence. We consider the subspace
\begin{equation*}
\pi\left(\mathscr{L}^{1}(G, \omega)\right)H_\pi:=\operatorname{span}\left\{\pi(f) \xi\ :\ f \in \mathscr{L}^{1}(G, \omega),\ \xi \in H_ \pi\right\}     \end{equation*}
which is dense $H_\pi,$ formed by sums of the form
\begin{equation*}
   \sum_{i=1}^{n} \pi(f_i) \xi_i \quad \text { for } f_i \in \mathscr{L}^{1}(G, \omega) \text { y } \xi_i \in H_\pi. 
\end{equation*}
We define 
\begin{equation*}
  \tilde{\pi}(s)\left(\sum_{i=1}^{n} \pi(f_{i}) \xi_i\right):=\sum_{i=1}^{n} \pi\left(\Gamma_{\omega}^{s} f_i\right) \xi_i.  
\end{equation*}
To prove that $\tilde{\pi}(s)$  is well-defined, we will show that if $\sum_{i=1}^{n} \pi(f_i) \xi_i=0$, then $\sum_{i=1}^{n} \pi\left(\Gamma_{\omega}^{s} f_i\right) \xi_i=0$ for each $s \in G.$
By the Lemma $\ref{gama y theta brincan adjuntos }$ we have that for each $s \in G$ and $f, g \in \mathscr{L}^{1}(G, \omega),$ the following holds:
\begin{equation*}
 g^{*} \ast_\omega \Gamma_{\omega}^{s} f=\left(\Gamma_{\omega}^{s-1} g\right)^{*} \ast_\omega f,   
\end{equation*}
and thus, for $\xi,\eta \in H_{\pi}$ y $f_1, \ldots, f_n \in \mathscr{L}^{1}(G, \omega) $ we have that
\begin{align*}
 \left\langle\sum_{i=1}^{n} \pi\left(\Gamma_{\omega}^{s} f_i\right)\xi, \pi(g) \eta\right\rangle
& =\sum_{i=1}^{n}\left\langle\pi\left(g^{*} \ast_\omega \Gamma_{\omega}^{s} f_i\right) \xi, \eta\right\rangle \\
& =\sum_{i=1}^{n}\left\langle\pi\left(\left(\Gamma^{ s^{-1}}_\omega g\right)^{*} \ast_\omega  f_i\right) \xi, \eta\right\rangle \\
& =\sum_{i=1}^{n}\left\langle\pi\left(\Gamma_{\omega}^{s^{-1}} g\right)^{*} \pi(f_i) \xi, \eta\right\rangle \\
& =\left\langle\sum_{i=1}^{n} \pi(f_i) \xi, \pi\left(\Gamma_\omega^{s^{-1}} g\right) \eta\right\rangle.
\end{align*}
That is,
\begin{align}\label{prod int}
 \left\langle\sum_{i=1}^{n} \pi\left(\Gamma_{\omega}^{s} f_i\right)\xi, \pi(g) \eta\right\rangle=\left\langle\sum_{i=1}^{n} \pi(f_i) \xi, \pi\left(\Gamma_\omega^{s^{-1}} g\right) \eta\right\rangle.    \end{align}
Now, to prove that  $\tilde{\pi}(s)$ is well-defined, suppose that
\begin{equation*}
 \sum_{i=1}^{n} \pi\left(f_{i}\right) \xi_i=0,
\end{equation*}
by $\eqref{prod int}$, we have that
\begin{equation*}
    0=\left\langle 0, \pi\left(\Gamma_{\omega}^{s-1} g\right) \eta\right\rangle=\left\langle\sum_{i=1}^{n} \pi\left(\Gamma_{\omega}^{s} f_i\right) \xi_i, \pi(g) \eta\right\rangle.
\end{equation*}
Thus, the vector $\sum_{i=1}^{n} \pi\left(\Gamma_{\omega}^{s} f_i\right)\xi_i$is orthogonal to all vectors of the form $\pi(g) \eta$ for each $\eta\in H_\pi$ and since  
\begin{equation*}
\overline{\operatorname{span}}\{\pi(f) \xi\ :\ f \in \mathscr{L}^{1}(G, \omega),\ \xi \in H_ \pi\}=H_\pi    
\end{equation*}
 it follows that  $\sum_{i=1}^{n} \pi\left(\Gamma_{\omega}^{s} f_i\right) \xi_i=0.$ Moreover, a straightforward application of $\eqref{prod int}$ shows that the operator $\tilde{\pi}(s)$  is unitary in the space $\pi\left(\mathscr{L}^{1}(G, \omega)\right) H_{\pi}$ and then, the operator $\tilde{\pi}(s)$ it extends to a unique unitary operator on  $H_\pi$ with inverse $\tilde{\pi}\left(s^{-1}\right)$ ; furthermore, by Lemma \(\ref{gamma es morfismo}\), it follows that

\begin{align*}
\tilde{\pi}(s r)  \left(\sum_{i=1}^{n} \pi(f_i) \xi_i\right)&=\sum_{i=1}^{n} \pi\left(\Gamma_{\omega}^{s r} f_i\right) \xi_i \\
&= \sum_{i=1}^{n} \pi\left(\Gamma_{\omega}^{s} \Gamma_{\omega}^{r} f_i\right) \xi_i\\
&=\tilde{\pi}(s)\left(\sum_{i=1}^{n} \pi\left(\Gamma_{\omega}^{r} f_i\right) \xi_i\right) \\
&=  \tilde{\pi}(s) \tilde{\pi}(r)\left(\sum_{i=1}^{n} \pi(f_i) \xi_i\right).\\
\end{align*}
Thus, $\tilde{\pi}(s r)=\tilde{\pi}(s) \tilde{\pi}(r)$ for each $s, r \in G.$ On the other hand, by Lemma $\ref{Gamma SOT continuo }$ for each $f \in \mathscr{L}^{1}(G, \omega)$ the mapping  $s\mapsto \Gamma_{\omega}^{s} f$ is continuous. It follows that $s \mapsto \tilde{\pi}(s) \xi$ is continuous for all $\xi \in H_{\pi}$, and thus $\left(\pi, H_{\pi}\right)$ is a unitary representation of $G$. Applying the previous theorem  we can induce a non-degenerate $*$-representation of $\mathscr{L}^{1}(G, \omega)$, also denoted by $\tilde{\pi}$.
To prove the equation 
\begin{equation*}
\langle\pi(f) \xi, \eta\rangle=\int_{G} f(s)\langle\tilde{\pi}(s) \xi, \eta\rangle \omega(s) d \mu(s)    
\end{equation*}
we will to show $\pi(f)=\tilde{\pi}(f)$ for each $f \in \mathscr{L}^{1}(G, \omega)$ and by continuity it is sufficient to show that
$\langle\pi(f) \pi(g) \xi, \eta\rangle=\langle\tilde{\pi}(f) \pi(g) \xi, \eta\rangle$ for each $f, g \in C_{C}(G)$ and $\xi, \eta \in H_{\pi}.$ On the other hand, as in the usual case, for more details, see \cite[Lemma B.6.5]{Deitmar} it holds that the Bochner integral $\int_G f(s)\omega(s)\Gamma_\omega^sg d\mu(s)$ exists in the Banach space $\mathscr{L}^1(G,\omega)$ and coincides with  $f\ast_\omega g.$ 
Note that
\begin{align*}
\langle\tilde{\pi}(f) \pi(g) \xi, \eta\rangle & =\int_G f(s)\langle\tilde{\pi}(s)(\pi(g) \xi), \eta\rangle \omega(s) d s \\
& =\int_G f(s)\left\langle\pi\left(\Gamma_{\omega}^{s} g\right) \xi, \eta\right\rangle \omega(s) d s \\
& =\left\langle\pi\left(\int_G f(s) \left(\Gamma_{\omega}^{s} g\right) \omega(s) d s\right) \xi, \eta\right\rangle \\
& =\langle\pi(f \ast_\omega g) \xi, \eta\rangle \\
& =\langle\pi(f) \pi(g) \xi, \eta\rangle
\end{align*}
it follows that $\tilde{\pi}(f) \pi(g)=\pi(f) \pi(g)$ and thus $\tilde{\pi}(f)=\pi(f)$. To complete the proof, suppose that  $\tilde{\rho}$ is another unitary representation of $G$  such that for each $f\in\mathscr{L}^1(G,\omega)$ we have
\begin{align*} 
\langle \pi(f)\xi, \eta \rangle = 
\int_G f(s) \langle \tilde{\rho}(s)\xi, \eta \rangle d\mu(s).
\end{align*}
Then $\tilde{\rho}$ induces a $*$-representation 
of $\mathscr{L}^1(G,\omega)$, also denoted by $\tilde{\rho}$  and we have $\tilde{\rho}=\pi$ 
on $\mathscr{L}^1(G,\omega)$. By   $\eqref{pi(f)}$ it follows that $\langle\tilde{\pi}(s)\xi,\eta\rangle=\langle\tilde{\rho}(s)\xi,\eta\rangle$ for each $s\in G$ y $\xi,\eta\in H_\pi$. Therefore, $\tilde{\pi}(s)=\tilde{\rho}(s)$ for all $s\in  G.$\qedhere
\end{proof}
As mentioned earlier, the mapping $s\to \Gamma_\omega^s$  corresponds to the left translation; however, $\Gamma_\omega^s: L^2(G, \omega) \to L^2(G, \omega)$ is not a unitary operator. To address this issue, we consider the spaces for $1 \leq p < +\infty$.
\begin{equation*}
\mathcal{L}^p(G,\omega):=\{ f\in\mathcal{B}(G)\ :\ \int_G |f(s)|^p\omega(s)^pd\mu(s)<+\infty\}   
\end{equation*} with the norm
\begin{equation*}
 \|f\|_{\omega,p}:=\left(\int_G |f(s)|^p\omega(s)^pd\mu(s)\right)^\frac{1}{p}.   
\end{equation*}
In particular, for $p=2$ with the inner product $\langle f,g\rangle:=\int_G f(s)\overline{g(s)}\omega^2(s)d\mu(s)$ $\mathcal{L}^2(G,\omega)$ is a Hilbert space. Considering the operator $\Gamma_\omega^s$ over $\mathcal{L}^2(G,\omega)$ we obtained the following result analogous to the Lemma $\ref{Gamma SOT continuo }$
\begin{lema}\label{Gamma SOT continuo 2 }
Let $f\in \mathcal{L}^p(G,\omega)$ be fixed. Then, the mapping $s\mapsto \Gamma_\omega^s f$  from $G$ to $\mathcal{L}^p(G,\omega)$ is continuous.
\end{lema}
\begin{proof}
Note that
\begin{align*}
\|\Theta^s_\omega f-\Theta^r_\omega f\|_{\omega,p}^p&=\int_G\frac{1}{\omega(t)^p}|R_s(f\omega)(t)-R_r(f\omega)(t)|^p\omega(t)^pd\mu(t)\\
&=\int_G|R_s(f\omega)(t)-R_r(f\omega)(t)|^pd\mu(t)\\
&=\|R_s(f\omega)-R_r(f\omega)\|_p^p\to 0.
\end{align*}
This since $f\omega\in L^p(G)$ and by a straightforward application of \cite[proposition 2.42]{Folland}.
\end{proof}
\begin{lema}\label{operadores} 
Let $G$ be a locally compact abelain group, $\omega$ a weight on $G$ and $f,g\in \mathcal{L}^p(G,\omega).$ Then
\begin{enumerate}
    
    \item $\|\Gamma_\omega^sf\|_{\omega,p}=\|f\|_{\omega,p}$\hspace{0.1cm}  y \hspace{0.1cm} $\|\Theta_\omega^sf\|_{\omega,p}=\|f\|_{\omega,p},$
    \item  $\Theta_\omega^s\Gamma_\omega^sf=f=\Gamma_\omega^s\Theta_\omega^sf.$
\end{enumerate}
 \end{lema}
 \begin{proof}
Note that given $f,g\in \mathcal{L}^p(G,\omega)$ and $s\in G$   it follows that
  \begin{enumerate}

\item We have that 
\begin{align*} 
\|\Theta_\omega^sf\|_{\omega,p}&=\left(\int_G |\Theta_\omega^sf(t)|^p\omega(t)^pd\mu(s)\right)^{\frac{1}{p}}\\
&=\left(\int_G \left|\frac{f(ts)\omega(ts)}{\omega(t)}\right|^p\omega(t)^pd\mu(s)\right)^{\frac{1}{p}}\\
&=\left(\int_G |f(ts)|^p\omega(ts)^pd\mu(s)\right)^{\frac{1}{p}}\\
&=\|f\|_{\omega,p}.
\end{align*}
Similarly, it follows that $\|\Gamma_\omega^sf\|_{\omega,p}=\|f\|_{\omega,p}.$
\item 
The proof is a straightforward calculation.\qedhere
  \end{enumerate}
 \end{proof}
\begin{obs}\label{unitarios}
From the Lemma $\ref{operadores}$ it follows that for each $s\in G,$ the operator $\Gamma_\omega^s:\mathcal{L}^2(G,\omega)\to \mathcal{L}^2(G,\omega)$ is a surjective isometry; that is, a unitary operator.
\end{obs}

\begin{teo}
Let $G$ be a locally compact abelian group, then the mapping $s\mapsto\Gamma_\omega^s$ form $G$ to $\mathcal{B}(\mathcal{L}^2(G,\omega))$ is a uniary representation from $G.$ 
\end{teo}
\begin{proof}
It follows from the Observation $\ref{unitarios},$  Lemma $\ref{Gamma SOT continuo 2 }$ and  Lemma $\ref{gamma es morfismo}.$
\end{proof}

%\begin{Backmatter}

\bibliographystyle{plain}
\bibliography{Bibliografia.bib}

%\printaddress

%\end{Backmatter}

\end{document}